\documentclass[a4paper, oneside, british, 11pt]{article}

\usepackage{geometry} 
\usepackage{fullpage} 
\usepackage{fontspec} 
\usepackage{babel} 
\usepackage{float} 
\usepackage{authblk} 
\usepackage{lipsum} 
\usepackage{csquotes} 
\usepackage{wrapfig} 
\usepackage{datetime} 

\usepackage{subcaption} 
\usepackage{multicol} 
\usepackage[bottom]{footmisc} 
\usepackage{todonotes} 
\usepackage{hyperref} 

\usepackage{tikz}
\usetikzlibrary{patterns,arrows}
\tikzset{shorten <>/.style={shorten >=#1, shorten <=#1}}

\usepackage{enumitem} 
\setlist[enumerate]{nosep, label=(\roman*), leftmargin=*}

\usepackage[backend=biber, giveninits=true]{biblatex} 
\usepackage{graphicx} 
\graphicspath{{images/}}

\usepackage{color} 

\usepackage{amsmath, amssymb, amsfonts}
\usepackage{mathtools, stmaryrd}
\usepackage{tikz-cd}
\usepackage{mathUtil}

\usepackage{amsthm}

\renewcommand{\Box}{\boxempty}

\newcommand{\mbf}{\mathbf}

\newcommand{\msf}{\mathsf}

\makeatletter
\title{Fischer-Servi logic does not have interpolation}
\author[1]{Rodrigo Nicolau Almeida}
\author[1]{Nick Bezhanishvili}
\author[1,2]{Simon Lemal}
\affil[1]{Institute for Logic, Language and Computation, Universiteit van Amsterdam}
\affil[2]{Department of Mathematics, Université du Luxembourg}
\date{}
\makeatother

\begin{document}

\maketitle

\begin{abstract}
We prove that the Fischer-Servi logic $\mathsf{IK}$ does not have the (Craig) interpolation property. This is  obtained by showing that the corresponding class of modal Heyting algebras lacks the amalgamation property. We also generalize this result to some extensions of the Fischer-Servi logic
such as $\mathsf{IT}$, $\mathsf{IK4}$, $\mathsf{IS4}$, and 
$\mathsf{IGL}$.
\end{abstract}

\section{Introduction}

The Craig interpolation property (CIP) is an important property of logical systems. For propositional logics, interpolation is closely connected to the notion of (super)amalgamation of the corresponding algebraic models of a given logic. We refer to the recent handbook on interpolation \cite{HandbookInterp26} for the current state of the art on Craig interpolation, as well as to Chapters \cite{BtCI26, Fus26, Met26} of this handbook for detailed discussions of the connections between (super)amalgamation and Craig interpolation.

Maksimova’s classical result \cite{Maksimova1977} states that there are exactly seven consistent superintuitionistic logics with Craig interpolation. This result was proved by characterizing all seven nontrivial varieties of Heyting algebras with the amalgamation property (in varieties of Heyting algebras, the notions of amalgamation and superamalgamation coincide); for the proof we refer, for example, to \cite[Section 14.4]{Chagrov1997-cr}. Using similar techniques, Maksimova \cite{Maks1979} also showed that there are at most 37 modal logics extending $\mathsf{S4}$ that have the Craig interpolation property (see, e.g., \cite{Chagrov1997-cr, Fus26, GM05}).

Alongside superintuitionistic and classical modal logics, intuitionistic modal logics play an important role (see e.g., \cite{paivaartemov}). One of the first, and still most influential, systems of intuitionistic modal logic was introduced by Fischer-Servi \cite{fischerservi1984}. This system was later studied by Simpson \cite{Simpson1994} and  Wolter and Zakharyaschev \cite{Wolter1997IntuitionisticML}, among others. G.~Bezhanishvili \cite{Bezhanishvili1998,Bezhanishvili1999,Bezhanishvili2000} investigated the intuitionistic modal logic $\mathsf{MIPC}$, which provides the one-variable fragment of intuitionistic predicate calculus in the same way that $\mathsf{S5}$ provides the one-variable fragment of classical predicate logic \cite{Halmos1954-1956}.

Interpolation in intuitionistic modal logic has recently attracted renewed attention. In this paper we adopt the convention used by van der Giessen \cite{vandergiessenthesis} and collaborators, of denoting by
\begin{enumerate}
    \item $\mathsf{iL}=\mathsf{IPC}\oplus \mathsf{L}_{\Box}$, the logic with $\Box$-only and $\Box$-only axioms, for some classical modal logic $\mathsf{L}$;
    \item $\mathsf{CL}=\mathsf{IPC}\oplus \mathsf{L}_{\Box,\Diamond}$, the logic with both $\Box$ and $\Diamond$, and axioms given for both modalities over $\mathsf{L}$;
    \item  $\mathsf{IL}=\mathsf{IK}\oplus \mathsf{L}_{\Box,\Diamond}$ the extension of Fischer-Servi's basic logic $\mathsf{IK}$ with axioms in both modalities (except when otherwise noted!).
\end{enumerate}

For clarity, we consider only  \emph{normal} intuitionistic modal logics. We give a short summary of the results concerning interpolation properties known to us at the time of writing this paper.\footnote{The reader (in 2026) can find in Taishi Kurahashi's frequently updated  \href{https://www2.kobe-u.ac.jp/~tk/jp/notes/ULIP.html}{webpage}, a summary of most known results on Craig interpolation and its uniform and Lyndon variants for intuitionistic, modal logics and provability logics. We also note that the statements below are given for CIP, though many of the stated logics also enjoy some of these stronger properties.}:
\begin{enumerate}
    \item The logics $\mathsf{iK}$ and $\mathsf{iD}$ have CIP \cite[Theorem 34]{Iemhoff2019}. The logic $\mathsf{iS4}$ also has CIP \cite{CZERMAK1975381}, as do many other related systems  \footnote{See the appendix for a short proof, relying on standard techniques from algebra and duality theory, of the fact that $\mathsf{iS4},\mathsf{iT},\mathsf{iK4}$, all have CIP. This was mentioned as a conjecture, e.g., in \cite[Example 2.2.2]{vandergiessenthesis}.}.
    \item In the family of ``intuitionistic provability logics", the logics $\mathsf{iSL}$ \cite{Fre2024} and $\mathsf{iGL}$\cite{vanderGiessen2021} have CIP; the logic $\mathsf{KM}$ (see \cite{Muravitsky2014}) also has CIP.
    \item Among the ``lax logic" family of logics, the Lax Logic has CIP \cite{Iemhoff2024}.
    \item The logics $\mathsf{CK}$ and $\mathsf{WK}=\mathsf{CK}\oplus \neg\Diamond\bot$ were recently shown to have CIP \cite{vandergiessenshilito}.
\end{enumerate}

We also remark that prior to much of this literature,  G.~Bezhanishvili and Marx showed (unpublished, early 2000s) that $\mathsf{MIPC}$ does not have the interpolation property. This result was obtained using techniques from Areces and Marx \cite{Marx1998-MAAFOI} and Marx \cite{Marx1998}, by showing that the corresponding variety of monadic Heyting algebras does not have the amalgamation property.

In this paper, we continue this line of work, showing that the Fischer-Servi logic and some of its close neighbours, do not have the Craig interpolation property. This stands in contrast to the above abundance of $\mathsf{iL}$-logics with the interpolation property, showing that the addition of the diamond and the compatibility conditions poses challenges to the interpolation property. We prove this by demonstrating that the corresponding variety of modal Heyting algebras does not have the amalgamation property. This, in turn, is established using the duality between modal Heyting algebras and relational Esakia spaces (originally obtained by Palmigiano \cite{PalmigianoDualities2004}). Using this duality, we construct a diagram of finite relational Kripke frames that does not co-amalgamate. Consequently, the corresponding modal Heyting algebras do not amalgamate, and hence the Fischer-Servi logic lacks the interpolation property. We also show that, with slight modifications, this example can be extended to some other intuitionistic modal logics of interest:
\begin{enumerate}
    \item The logics $\mathsf{IT}$, $\mathsf{IK4}$ and $\mathsf{IS4}$, introduced by Simpson \cite{Simpson1994} and studied e.g. in \cite{Girlando2023}.
    \item The logic $\mathsf{IGL}$ introduced in \cite{IGLlogicdasmarinvan} and studied e.g. in \cite{Aguilera2025}.
\end{enumerate}

We conclude with some possible directions in the study of interpolation in extensions of $\mathsf{IK}$.

\section{Super-Fischer-Servi logics, FS-algebras and IK-frames}

Let $\mathsf{IPC}$ denote the intuitionistic propositional  calculus. We recall the definition of $\mathsf{IK}$ logic, as given in, e.g., \cite{Simpson1994}:

\begin{definition}
    A set of formulas $L$ in the language $\mathcal{L}_{\Box,\Diamond}=(\wedge,\vee,\rightarrow,\Box,\Diamond,\bot,\top)$, is called a \emph{super-Fischer-Servi logic} (super FS-logic), if it contains the axioms:
    \begin{enumerate}
        \item $\mathsf{IPC}\subseteq L$;
        \item $\Box(p\rightarrow q)\rightarrow (\Box p\rightarrow \Box q)$;
        \item $\Box(p\rightarrow q)\rightarrow (\Diamond p\rightarrow \Diamond q)$;
        \item $\Diamond(p\vee q)\rightarrow (\Diamond p\vee \Diamond q)$;
        \item $(\Diamond p\rightarrow\Box q) \rightarrow \Box(p\rightarrow q)$
        \item $\Diamond\bot\rightarrow\bot$,
    \end{enumerate}
    and is closed under Uniform Substitution, Modus Ponens and Necessitation. We write $\mathsf{IK}$ for the smallest super-FS logic. We write $\phi\in L$ to mean that $\phi$ is a \emph{theorem} of $L$.
\end{definition}

An equivalent axiomatization over intuitionistic logic, which will be more convenient for our purposes, was given in \cite{Wolter1999}: 

\vspace{2mm}

\begin{enumerate}
    \item The normality axioms $\Box(p\wedge q)\leftrightarrow \Box p\wedge \Box q$ and $\Box\top=\top$ and $\Diamond(p\vee q)\leftrightarrow \Diamond p\vee \Diamond q$, and $\Diamond\bot\leftrightarrow\bot$;
    \item The connection axioms $\Diamond(p\rightarrow q)\rightarrow (\Box p\rightarrow \Diamond q)$ and $(\Diamond p\rightarrow \Box q)\rightarrow \Box(p\rightarrow q)$.
\end{enumerate}


\subsection{FS-algebras}

We recall the notion of $\mathbf{FS}$-algebras:

\begin{definition}
    A tuple $(H,\Box,\Diamond)$ is called an $\mathbf{FS}$-algebra if $H$ is a Heyting algebra and $\Box, \Diamond: H \to H$ are maps satisfying  the following axioms for each $a, b\in H$:

\vspace{2mm}
    
    \begin{enumerate}
        \item $\Box 1=1$ and $\Diamond 0=0$;
        \item $\Box(a\wedge b)=\Box a\wedge\Box b$ and $\Diamond (a\vee b)=\Diamond a\vee \Diamond b$;
        \item $\Diamond (a\rightarrow b)\leq \Box a\rightarrow \Diamond b$;
        \item $\Diamond a\rightarrow \Box b\leq \Box(a\rightarrow b)$.
    \end{enumerate}

    \vspace{2mm}
    
    We denote by $\mathbf{FS}_{\mathrm{alg}}$ the category of $\mathbf{FS}$-algebras and $\Box$ and $\Diamond$-preserving Heyting homomorphisms.
\end{definition}

Given an $\mathbf{FS}$-algebra $A$, and a formula $\phi(\overline{p})$, we can extend an assignment $v\colon \overline{p}\to A$, in the usual way to a map $\overline{v}\colon \mathrm{Form}_{\mathcal{L}_{\Box,\Diamond}}(\overline{p})\to A$. We write $(A,v)\vDash \phi(\overline{p})$
 to mean that $\overline{v}(\phi)=1$. We write $A\vDash \phi$ to mean that this holds for any assignment.

Given $L$ a super FS-logic, we generically denote by $(\mathbf{FS}_{\mathrm{L}})_{\mathrm{alg}}$ the category of $\mathbf{FS}$-algebras which additionally validate all the axioms from $L$, i.e., if for each $\phi\in L$, and each $A\in (\mathbf{FS}_{\mathrm{L}})_{\mathrm{alg}}$ we have $A\vDash \phi$.

\begin{notation}
    Given $X$ a set, and $L$ a super FS-logic, we denote by $\mathcal{F}_{L}(X)$ the free $\mathbf{FS}_{\mathrm{L}}$-algebra on $X$ many generators. This is precisely the quotient of $\mathrm{Form}_{\mathcal{L}_{\Box,\Diamond}}$ given by $\phi\sim \psi$ if and only if $\phi\leftrightarrow\psi\in L$.
\end{notation}

The following is then the basic algebraic completeness, following from the general theory of algebraizability \cite[Theorem 2.12]{Font2016-dk} (see also \cite[Chapter 8]{Chagrov1997-cr} for the analogous cases of intuitionistic and classical modal logics): 

\begin{theorem}\label{thm: algebraic completeness theorem}
    Let $L$ be a super FS-logic.
    For each formula $\phi(\overline{p})\in \mathcal{L}_{\Box,\Diamond}$ the following are equivalent:
    \begin{enumerate}
        \item $\phi\in L$;
        \item For any $\mathbf{FS}_{\mathrm{L}}$-algebra $A$, we have $A\vDash \phi$.
        \item In $\mathcal{F}_{L}(\overline{p})$, we have that $[\phi]=1$.
    \end{enumerate}
\end{theorem}

\subsection{Fischer-Servi frames and spaces}

The  definitions in this section can be found in  \cite{PalmigianoDualities2004}.

\begin{definition}
    Let $F = (X , \leq ,R)$ be a tuple where $(X, \leq)$ is a poset and $R$ a binary relation. We say that $F$ is an \emph{$\mathbf{FS}$-frame} provided the following conditions are satisfied:

    \begin{itemize}
        \item \textbf{F1}:  $x' \geq xRy \implies \exists y'.x'Ry'\geq y$;
        \item \textbf{F2}:  $xRy\leq y' \implies \exists x'. x \leq x'Ry'$.
    \end{itemize}
\end{definition}

By convention, horizontal arrows will correspond to $R$, and vertical arrows to $\leq$. We will also use colors to assist in the distinction: $R$ will be red, and $\leq$ will be blue.

\begin{figure}[h]
    \centering
\begin{tikzpicture}
     \node at (0,0) {$x$};
     \node at (0,2) {$x'$};
     \node at (2,2) {$y'$};
     \node at (2,0) {$y$};
     \node at (1,-0.5) {$\mathbf{F1}$};

     \draw[->, blue, thick] (0,0.3) -- (0,1.7);
     \draw[->, red, thick] (0.3,0) -- (1.7,0);
     \draw[->,dotted,red] (0.3,2) -- (1.7,2);
     \draw[->,dotted, blue] (2,0.3) -- (2,1.7);
    \end{tikzpicture}
\qquad
\begin{tikzpicture}
     \node at (3,0) {$x$};
     \node at (3,2) {$x'$};
     \node at (5,2) {$y'$};
     \node at (5,0) {$y$};
     \node at (4,-0.5) {$\mathbf{F2}$};

     \draw[->,dotted, blue] (3,0.3) -- (3,1.7);
     \draw[->, red, thick] (3.3,0) -- (4.7,0);
     \draw[->,dotted, red] (3.3,2) -- (4.7,2);
     \draw[->, blue, thick] (5,0.3) -- (5,1.7);
 \end{tikzpicture}
    \caption{Confluence conditions for Fischer-Servi}
    \label{fig:placeholder}
\end{figure}
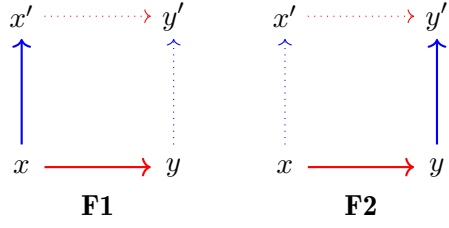

We will also need the morphisms of Fischer-Servi frames:

\begin{definition}
Let $f\colon X\to Y$ be a map between $\mathbf{FS}$-frames. We say that $f$ is an $\mathbf{FS}$-morphism if:
\begin{enumerate}
\item $f$ is a $\leq$-p-morphism.
\item If $xRy$, then $f(x)Rf(y)$.
\item If $f(x)Rz$, then $z\leq f(x')$ for some $xRx'$;
\item If $f(x)\leq m$ and $mRz$, then there is some $x\leq x'$ and $x'Rx''$ and $f(x'')\leq z$.
\end{enumerate}
\end{definition}

We refer to (iii) as the ``weak back condition" and (iv) as the ``strong back condition". Recall that Esakia spaces are order-topological spaces dual to Heyting algebras \cite{Esakiach2019HeyAlg}. We will mostly be concerned with finite Esakia spaces, which are equivalent to finite posets; the reader unfamiliar with this duality can safely consider such finite posets as the object of interest in this paper, with the sole exception of Lemma \ref{lem: FS spaces are FS frames} below.

\begin{notation}
    Given a frame $(X,S)$, where $S\subseteq X\times X$ is any relation, and $U\subseteq X$ is any subset, we denote the following subsets:    \begin{equation*}
        \Diamond_{S}U=\{x\in X : \exists y, xRy \wedge y\in U\} \text{ and } \Box_{S}U=\{x\in X : \forall y(xRy \rightarrow y\in U\}.
    \end{equation*}
\end{notation}

\begin{definition}
    A tuple $(X,\leq,R,\tau)$ is called an $\mathbf{FS}$-space\footnote{In \cite{PalmigianoDualities2004}, these are called $\mathbf{IK}$-spaces, and the associated morphisms are called $\mathbf{IK}$-morphisms.} if $(X,\leq,\tau)$ is an Esakia space, and also:
    \begin{enumerate}
    \item $R[x]=R[{\uparrow}x]\cap {\downarrow}R[x]$
        \item For each $x\in X$, $R[x]$ is closed, and $R[{\uparrow}x]$ is a closed upset.
        \item Whenever $U$ is a clopen upset, $\Diamond_{R}U$ is a clopen upset, and $\Box_{\leq\circ R}U$ is a clopen upset.
    \end{enumerate}
    We denote by $\mathbf{FS}_{\mathrm{sp}}$ the category of $\mathbf{FS}$-spaces and continuous $\mathbf{FS}$-morphisms.
\end{definition}

We stress that $\mathbf{FS}$-spaces are also $\mathbf{FS}$-frames, and finite $\mathbf{FS}$-frames are $\mathbf{FS}$-spaces:

\begin{lemma}\label{lem: FS spaces are FS frames}\cite[Lemma 4.1.1.3]{PalmigianoDualities2004}
Let $(X,\leq,R,\tau)$ be a tuple, where $(X,\leq,\tau)$ is an Esakia space, and $R\subseteq X\times X$. Then:
\begin{enumerate}
    \item \label{eq: finite part} If $X$ is finite, and it satisfies $R[x]=R[{\uparrow}x]\cap {\downarrow}R[x]$, then $(X,\leq,R)$ is an $\mathbf{FS}$-frame if and only if $(X,\leq,R,\tau)$ is an $\mathbf{FS}$-space.
    \item \label{eq: infinite part} If $(X,\leq,R,\tau)$ is an $\mathbf{FS}$-space, then $(X,\leq,R)$ is an $\mathbf{FS}$-frame.
\end{enumerate}
\end{lemma}

The following was shown in \cite[Theorem 6.1.11]{PalmigianoDualities2004}:

\begin{theorem}
    The categories $\mathbf{FS}_{\mathrm{alg}}$ and $\mathbf{FS}_{\mathrm{sp}}$ are dually equivalent.
\end{theorem}

Given an $\mathbf{FS}$-algebra $H$, we denote by $\mathsf{Spec}(H)$ its dual, and given homomorphisms $f\colon H\to H'$ by $f_{*}\colon \mathsf{Spec}(H')\to \mathsf{Spec}(H)$ the dual map where $f_{*}=f^{-1}[-]$. This duality restricts to Esakia duality on its $R$-free reduct, respectively, $\Box,\Diamond$-free reduct. Consequently we immediately have the following lemma, which will be the most we will need out of the dual perspective:

\begin{lemma}\label{lem: injective to surjective duality}
    A map $f\colon H\to H'$ between $\mathbf{FS}$-algebras is injective if and only if $f_{*}\colon \mathsf{Spec}(H')\to \mathsf{Spec}(H)$ is surjective.
\end{lemma}

\section{Craig interpolation and superamalgamation}

In this section we will show that Craig interpolation implies the amalgamation property for $\mathbf{FS}$-algebras. We make precise what we mean by this:

\begin{definition}
    We say that a super-FS logic $L$ has the \emph{Craig interpolation property} (CIP) if whenever $\phi(\overline{p},\overline{q})\rightarrow\psi(\overline{q},\overline{r})\in L$, then there is a formula $\chi(\overline{q})$ such that $\phi(\overline{p},\overline{q})\rightarrow\chi(\overline{q})\in L$ and $\chi(\overline{q})\rightarrow\psi(\overline{q},\overline{r})\in L$.
\end{definition}

\begin{definition}
    Let $\mathbf{M}$ be a variety of algebras. A \emph{V-formation} is a tuple of algebras $A, B_1, B_2\in\mathbf{M}$ and embeddings $h_i\colon A\to B_i$, as in Figure \ref{subfig:vform}. Such a variety $\mathbf{M}$ has the \emph{amalgamation property} if for every V-formation, there is an algebra $C$, and injective homomorphisms $p_i\colon B_i\to C$ such that $p_1\circ h_1 = p_2\circ h_2$, or in other words, such that the diagram in Figure \ref{subfig:amal} commutes. The lattice $C$ is called an \emph{amalgam}.

    \begin{figure}[h!]
        \centering
        \begin{subfigure}{.49\textwidth}
            \centering
            \begin{tikzcd}
                & \phantom{a} & \\
                B_1 & & B_2 \\
                & A \arrow[lu, "h_1", hook] \arrow[ru, "h_2"', hook] &                             
            \end{tikzcd}
            \caption{V-formation}
            \label{subfig:vform}
        \end{subfigure}
        \hfill
        \begin{subfigure}{.49\textwidth}
        \centering
            \begin{tikzcd}
                & C & \\
                B_1 \arrow[ru, "p_1", hook] & & B_2 \arrow[lu, "p_2"', hook] \\
                & A \arrow[lu, "h_1", hook] \arrow[ru, "h_2"', hook] &                             
            \end{tikzcd}
            \caption{Amalgam}
            \label{subfig:amal}
        \end{subfigure}
        \caption{Amalgamation of algebras}
    \end{figure}
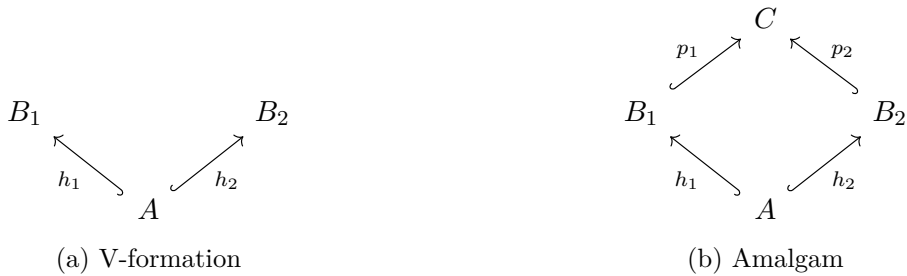
\end{definition}

To prove that Craig interpolation implies the amalgamation property, we need to establish a correspondence between filters and congruences of FS-algebras:

\begin{definition}
    A filter $F$ on an $\mathbf{FS}$-algebra is called a \emph{modal filter} if whenever $a\in F$ then $\Box a\in F$.
\end{definition}

\begin{lemma}
    Given an $\mathbf{FS}$-algebra $A$, there is a 1-1 correspondence between modal filters on $A$ and congruences.
\end{lemma}
\begin{proof}
    It is well-known (see e.g. \cite[Proposition 2.4.9]{Esakiach2019HeyAlg}) 
    that there is a 1-1 correspondence between filters and Heyting congruences, which is given by assigning to each filter $F$ the congruence
    \begin{equation*}
        \Theta(F)=\{(a,b) : a\leftrightarrow b\in F\},
    \end{equation*}
    and for each congruence $\Theta$ the filter
    \begin{equation*}
        F(\Theta)=\{a : (a,1)\in \Theta\}.
    \end{equation*}
We show that these assignments restrict to a bijection between modal filters and $\mathbf{FS}$-congruences. Indeed, let $F$ be a modal filter. We show that $\Theta(F)$ is a congruence. Indeed, assuming that $(a,b)\in \Theta(F)$, we have that $a\leftrightarrow b\in F$, so $a\rightarrow b\in F$, which since this is a modal filter, means that $\Box(a\rightarrow b)\in F$, and consequently using closure under $\wedge$ and upwards closure of $F$, $\Box a\rightarrow \Box b\in F$. Thus $\Box a\leftrightarrow \Box b\in F$, meaning that $(\Box a,\Box b)\in \Theta(F)$. The condition for $\Diamond$ follows similarly, since we have the axiom $\Box(a\rightarrow b)\rightarrow (\Diamond a\rightarrow \Diamond b)$.

Conversely, assuming that $\Theta$ is a congruence, if we assume that $a\in F(\Theta)$ then $(a,1)\in \Theta$, so $(\Box a,\Box 1)\in \Theta$, so by normality, $(\Box a,1)\in \Theta$, so $\Box a\in F$.
\end{proof}

We also have the following easy fact:

\begin{lemma}\label{lem: modal filter generated}
    Given an $\mathbf{FS}$-algebra $A$, and a set $S$, there is a smallest modal filter containing $S$. This is given by:
    \begin{equation*}
        \mathsf{Fil}_{\Box}(S)=\{b\in A :\exists a_{1},\dots,a_{n}\in S, \Box^{m_{1}}a_{1}\wedge\dots\wedge \Box^{m_{n}}a_{n}\leq b\},
    \end{equation*}
    where $\Box^{0}a=a$ and $\Box^{n+1}a=\Box\Box^{n}a$.
\end{lemma}

We can thus obtain the following analogue of Maksimova's theorem \cite{Maksimova1977}:

\begin{proposition}\label{prop: CIP implies amalgamation}
    For $L$ a super-FS logic, if $L$ has the Craig interpolation property, then $(\mathbf{FS}_{\mathrm{L}})_{\mathrm{alg}}$ has the amalgamation property.
\end{proposition}
\begin{proof}
    Let $(A,B_{1},B_{2})$ be a V-formation of $L$-algebras; for ease of notation, we identify $A$ with the subalgebra that is its image in $B_{1}$ and $B_{2}$, and without loss of generality suppose that that $A=B_{1}\cap B_{2}$ (since an amalgam of that V-formation would be an amalgam of the smaller one as well). We can present such algebras via quotients $q_{A}\colon \mathcal{F}_{L}(A)\to A$, $q_{B_{i}}\colon \mathcal{F}_{L}(B_{i})\to B_{i}$, the maps defined by sending the free generator $x_{b}$ for $b\in B_{i}$ to $b$ itself. We thus have that $\mathcal{F}_{L}(A)=\mathcal{F}_{L}(B_{1})\cap \mathcal{F}_{L}(B_{2})$ by our identification. Let $G=q_{A}^{-1}[1]$ be the filter-kernel, and similarly let $F_{i}=q_{B_{i}}^{-1}[1]$. Note that then $G=(F_{1}\cap \mathcal{F}_{L}(A))\cap (F_{2}\cap \mathcal{F}_{L}(A))$. Now consider $\mathcal{F}_{L}(B_{1}\cup B_{2})$, and $H=\mathsf{Fil}_{\Box}(F_{1}\cup F_{2})$. Note that because $F_{1}$ and $F_{2}$ are modal filters, from Lemma \ref{lem: modal filter generated} that
    \begin{equation*}
        H=\{c\in \mathcal{F}_{L}(B_{1}\cup B_{2}) : \exists a\in F_{1},b\in F_{2}, a\wedge b\leq c\}.
    \end{equation*}
    
Let $\iota_{i}\colon \mathcal{F}_{L}(B_{i})\to \mathcal{F}_{L}(B_{1}\cup B_{2})$ be the natural inclusion, and let $C=\mathcal{F}_{L}(B_{1}\cup B_{2})/H$ the quotient algebra. We can then define $p_{i}\colon B\to C$ by $[\iota_{i}(x_{b})]_{H}$. Note that this is well-defined: if $q_{B_{i}}(z)=q_{B_{i}}(z')$, then $z\leftrightarrow z'\in F_{i}$, and consequently by construction, $\iota_{i}(z)\leftrightarrow \iota_{i}(z')\in H$. Moreover, note that the diagram commutes by construction. We show that $p_{i}$ is injective, proving that $C$ is an amalgam.

Indeed, assume that for some $\phi(\overline{a},\overline{b}_{1})$, where $\overline{a}$ are elements of $A$ and $\overline{b}_{1}$ elements of $B_{1}$, we have $p_{1}(\phi)=1$. This means by construction that $\iota(\phi)\in H$, i.e. that there are $\psi(\overline{c}_{1},\overline{a}')\in F_{1}$ and $\mu(\overline{d}_{2},\overline{a}'')\in F_{2}$ such that $\psi\wedge \mu\leq \phi$. By basic properties of the implication, we have that then $\mu(\overline{d}_{2},\overline{a}'')\leq \psi(\overline{c}_{1},\overline{a}')\rightarrow \phi(\overline{b}_{1},\overline{a}) $. By the algebraic completeness (Theorem \ref{thm: algebraic completeness theorem}), this holds if and only if, by renaming variables as necessary:
\begin{equation*}
    \mu(\overline{p}_{2},\overline{a})\rightarrow(\psi(\overline{p}_{1},\overline{a})\rightarrow\phi(\overline{p}_{1},\overline{a}))\in L.
\end{equation*}
By Craig interpolation, let $\chi(\overline{a})$ be such that $\mu(\overline{p}_{2},\overline{a})\rightarrow\chi(\overline{a})\in L$, and $\chi(\overline{a})\rightarrow(\psi(\overline{p}_{1},\overline{a})\rightarrow\phi(\overline{p}_{1},\overline{a}))\in L$. Note that then $\chi(\overline{a})\in F_{2}$, and so by definition of $G$, $\chi(\overline{a})\in G$; by swapping some positions, $\chi(\overline{a})\rightarrow\phi(\overline{p}_{1},\overline{a})\in F_{1}$, which means since $\chi\in F_{1}$ that $\phi(\overline{p}_{1},\overline{a})\in F_{1}$, i.e., $[\phi]_{F_{1}}=1$. This shows injectivity of $p_{1}$, and the case for $p_{2}$ is entirely similar.
\end{proof}

\begin{remark}
    We note that in fact Craig interpolation implies the stronger \emph{superamalgamation property}, a property strictly stronger than amalgamation (see e.g. \cite[Section 7]{metcalfesuperinterpolation} for an extended discussion). Since our counterexamples are already on the failure of amalgamation, we did not include the proof -- following the standard arguments from Maksimova -- that CIP implies superamalgamation and that the converse likewise holds. 
\end{remark}

We end this section by corresponding the amalgamation property to a property on their dual $\mathbf{FS}$-spaces:

\begin{corollary}\label{cor: CIP implies coamalgamation}
    Let $L$ be a super-FS logic. Assume that $L$ has the Craig interpolation property. Then whenever $(X,Y_{1},Y_{2},f,g)$ is a tuple of $\mathbf{FS}_{\mathrm{L}}$-spaces such that $f\colon Y_{1}\to X$ and $g\colon Y_{2}\to X$ are surjective and continuous $\mathbf{FS}$-morphisms, then there is some $\mathbf{FS}_{\mathrm{L}}$-space $W$ and surjective and continuous $\mathbf{FS}$-morphisms $p_{1}\colon W\to Y_{1}$ and $p_{2}\colon W\to Y_{2}$ such that $fp_{1}=gp_{2}$.
\end{corollary}
\begin{proof}
    This follows immediately from Proposition \ref{prop: CIP implies amalgamation} and Lemma \ref{lem: injective to surjective duality}.
\end{proof}

We call a quintuple $(X,Y_{1},Y_{2},f,g)$ a \emph{co-V-formation}, and the space $W$ the \emph{co-amalgam}. The property of existence of co-amalgams for co-V-formations will thus be referred to as co-amalgamation.

\section{Failure of interpolation for $\mathsf{IK}$}

We now move on to introducing our counterexample to coamalgamation for the logic $\mathsf{IK}$. The three frames $X, Y, Z$ are denoted in Figure \ref{fig:coamalgamationfailure}. Explicitly, the blue arrows define the $\leq$-relations (with reflexive arrows omitted),  whilst the red arrows denote the $R$-relation. The maps $f$ and $g$ are defined as follows:
\begin{enumerate}
    \item $f(x_{i})=g(y_{i})=z_{i}$ for $i\in \{1,2\}$.
    \item $f(a_{0})=k_{0}=g(b_{1})=g(b_{0})$ and $g(b_{2})=k_{2}=f(a_{1})=f(a_{2})$.
    \item $f(a_{3})=k_{3}$ and $f(a_{4})=k_{4}$ and $g(b_{4})=k_{4}$ and $g(b_{3})=k_{3}$.
\end{enumerate}

\begin{figure}
    \centering
\begin{tikzpicture}
    \node (x0) at (0,0) {$x_{0}$};
    \node (x1) at (3,0) {$x_{1}$};
    \node (x2) at (0,3) {$x_{2}$};
    \node (a0) at (1.5,4.5) {$a_{0}$};
    \node (a1) at (0,4.5) {$a_{1}$};
    \node (a2) at (4.5,1) {$a_{2}$};
    \node (a3) at (1,2) {$a_{3}$};
    \node (a4) at (2,1) {$a_{4}$};
    \node at (2,-1) {$X$};

    \draw[->,red,thick] (x0) -- (x1);
    \draw[->,blue,thick] (x1) -- (a0);
    \draw[->,blue,thick] (x0) -- (x2);
    \draw[->,red,thick] (x2) -- (a0);
    \draw[->,red,thick] (x2) -- (a1);
    \draw[->,red,thick] (x2) -- (a1);
    \draw [->,blue,thick] (x1) -- (a2);
    \draw [->,blue,thick] (x0) -- (a3);
    \draw [->,blue,thick] (x0) -- (a4);
    \draw[->,red,thick] (a3) -- (a0);
    \draw[->,red,thick] (a4) -- (a2);

    \node (y0) at (8,0) {$y_{0}$};
    \node (y1) at (11,0) {$y_{1}$};
    \node (y2) at (8,3) {$y_{2}$};
    \node (b0) at (9.5,4.5) {$b_{2}$};
    \node (b1) at (8,4.5) {$b_{1}$};
    \node (b2) at (12.5,1) {$b_{0}$};
    \node (b3) at (9,2) {$b_{4}$};
    \node (b4) at (10,1) {$b_{3}$};
    \node at (10,-1) {$Y$};

    \draw [->,red,thick] (y0) -- (y1);
    \draw [->,blue,thick] (y0) -- (y2);
    \draw [->,red,thick] (y2) -- (b0);
    \draw [->,blue,thick] (y1) -- (b0);
    \draw [->,red,thick] (y2) -- (b1);
    \draw [->,blue,thick] (y1) -- (b2);
    \draw [->,blue,thick] (y0) -- (b3);
    \draw[->,red,thick] (b3) -- (b0);
    \draw [->,blue,thick] (y0) -- (b4);
    \draw[->,red,thick] (b4) -- (b2);

    \node (z2) at (4,-3) {$z_{2}$};
    \node (z0) at (4,-6) {$z_{0}$};
    \node (z1) at (7,-6) {$z_{1}$};
    \node (k0) at (5.5,-1.5) {$k_{0}$};
    \node (k2) at (8.5,-4.5) {$k_{2}$};
    \node (k3) at (5,-4) {$k_{3}$};
    \node (k4) at (6,-5) {$k_{4}$};
    \node at (6,-7) {$Z$};

    \draw[->, blue,thick] (z0) -- (z2);
    \draw[->,red,thick] (z0) -- (z1);
    \draw[->,blue,thick] (z1) -- (k0);
    \draw[->,red,thick] (z2) -- (k0);
    \draw[->,blue,thick] (z1) -- (k2);
    \draw[->,red,thick] (z2) -- (k2);
    \draw[->,blue,thick] (z0) -- (k3);
    \draw[->,blue,thick] (z0) -- (k4);
    \draw[->,red,thick] (k3) -- (k0);
    \draw[->,red,thick] (k4) -- (k2);

\end{tikzpicture}
    \caption{The co-V formation $(Z,X,Y)$ }
    \label{fig:coamalgamationfailure}
\end{figure}
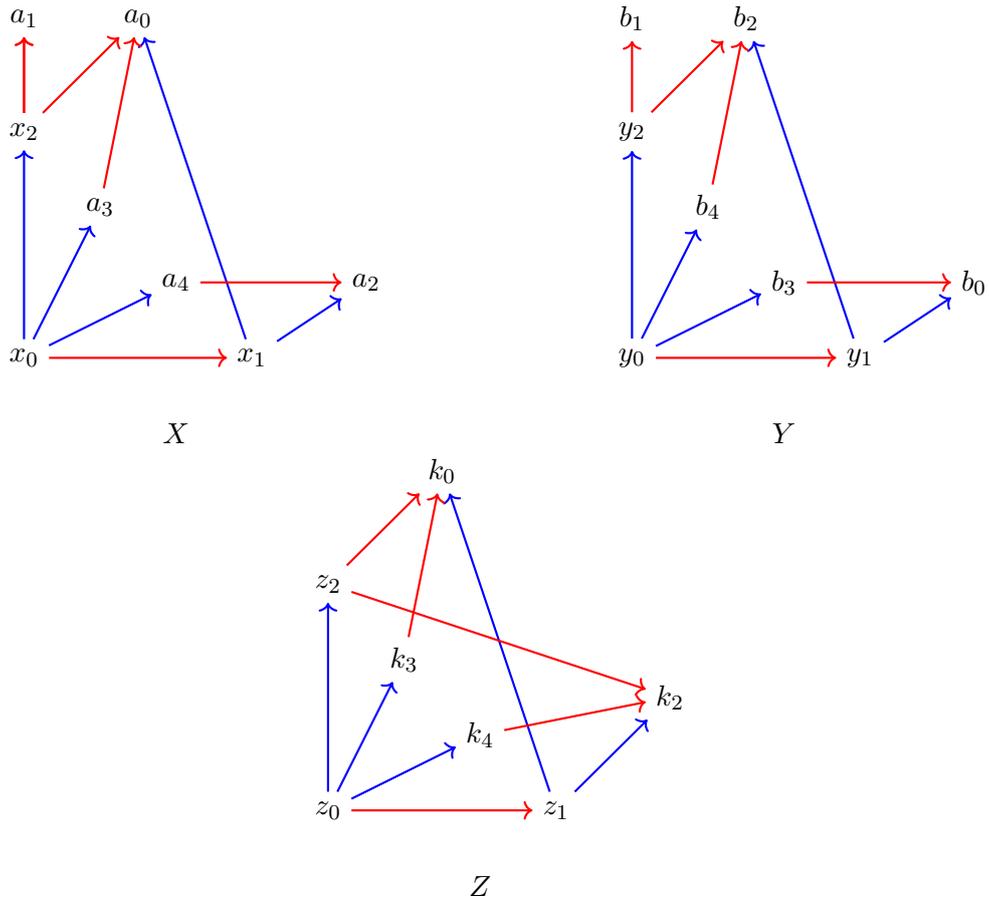

\vspace{2mm}

We will first verify that (1) $X,Y,Z$ are $\mathbf{FS}$-spaces, and  (2) second we will verify that the maps $f,g$ are surjective $\mathbf{FS}$-morphisms. We start with the former:

\begin{lemma}\label{lem: All the frames are IK}
    The frames $X,Y,Z$ are $\mathbf{FS}$-spaces.
\end{lemma}
\begin{proof}
By Lemma \ref{lem: FS spaces are FS frames}, it suffices to show that $X,Y,Z$ are $\mathbf{FS}$-frames and satisfy the condition $R[x]=R[{\uparrow}x]\cap {\downarrow}R[x]$. Note that $R[x]\subseteq R[{\uparrow}x]\cap {\downarrow}R[x]$ always. For the other direction, we need to show that whenever $z$ is a point, such that there exist $x\leq y$ where $yRz$ (so $z\in R[{\uparrow}x]$) and also $z\leq t$ where $xRt$ (so $z\in {\downarrow}R[x]$), then $xRz$. Note that in order for this to hold, there must be a succession of relations, $x\leq yRz\leq t$. In all of our diagrams, such a sequence of relations can only happen if one of the $\leq$ is reflexive, i.e., either $x=y$ or $z=t$; in the former case, because $yRz$, then $xRz$; in the latter case, by assumption, $xRt$ so $xRz$. This shows the condition.

We make such a verification for $Z$, since the cases for $X$ and $Y$ -- which are identical -- involve only one more arrow, which as it will follow from the arguments, never plays any role. We note that these are certainly posets and relational structures, so we essentially need to show that they satisfy the confluence relations $\mathbf{F1}$ and $\mathbf{F2}$. For this, we need to show that for any triple $(a,b,c)$ of points in the shape of $\mathbf{F1}$ or $\mathbf{F2}$ we can complete the squares. which we call \emph{closing the square}.

    We first consider the case of $\mathbf{F1}$. Note that certainly if $w$ is a node such that $w$ is $\leq$-maximal, then it will always close off any such diagram by taking $w$ itself. On the other hand, if $w$ is not $\leq$-maximal, but has no $R$-successors, then no such square can arise. So the only node $w$ which has proper $\leq$-successors and $R$-successors is $z_{0}$. Also note that its only $R$-successor is $z_{1}$, so we consider all the options for $z_{0}\leq a$ whilst $z_{0}Rz_{1}$.
    \begin{enumerate}
        \item If we take $z_{0}\leq z_{2}$ or $z_{0}\leq k_{3}$, then $z_{2}Rk_{0}$ (respectively, $k_{3}Rk_{0}$) and $z_{1}\leq k_{0}$. 
        \item If we take $z_{0}\leq k_{4}$, then $k_{2}$ closes off the diagram.
    \end{enumerate}
    This shows that $X$ satisfies the confluence axiom $\mathbf{F1}$. Now for $\mathbf{F2}$, note that if $w$ is a node, and $wRk$, but $k$ is $\leq$-maximal, then no diagram needs to be closed off. Thus $z_{2},k_{3},k_{4},k_{0},k_{2}$ do not need to be closed off. Also if a node has no $R$-successors, again, no closure needs to happen. So $z_{1}$ does not need any closure. So again we are left just with $z_{0}$. This only has $z_{1}$ as its successor.
    \begin{enumerate}
        \item If we take $z_{0}Rz_{1}$, and $z_{1}\leq k_{2}$, then $k_{4}$ closes off the diagram.
        \item If we take $z_{0}Rz_{1}$ and $z_{1}\leq k_{0}$ then $k_{3}$ closes off the diagram.
    \end{enumerate}
    We thus have that all diagrams are closed off, showing that $Z$ (and likewise, $X,Y$) are Fischer-Servi frames.
\end{proof}

We likewise show the following:
\begin{lemma}\label{lem: fs morphisms}
    The map $f\colon X\to Z$ is an $\mathbf{FS}$-morphism
\end{lemma}
\begin{proof}
First we show that it is a $\leq$-p-morphism. This is clear to observe, since the only identification made by $f$ is sending $a_{1}$ to $k_{2}$, which is a maximal node. Similarly for $g$. Next, note that $f$ preserves the relation:
\begin{enumerate}
\item We have $x_{2}Ra_{1}$ and $f(x_{2})=z_{2}Rk_{2}=f(a_{1})$; similarly $x_{2}Ra_{0}$ and $x_{2}Ra_{2}$. Similarly for $g$.
\item We have $a_{3}Ra_{0}$ and $f(a_{3})=k_{3}Rk_{0}=f(a_{0})$. Similarly for $g$, except note that the swap between the indices relative to $f$ is done in a consistent way.
\item Finally $x_{0}Rx_{1}$ and certainly $f(x_{0})=z_{0}Rz_{1}=f(x_{1})$.
\end{enumerate}

Next we show that it satisfies the weak back condition.
\begin{enumerate}
\item We have that $f(a_{3})Rk_{0}$, but this is obviously witnessed by $a_{0}$.
\item We have that $f(x_{2})Rk_{2}$, and $x_{2}Ra_{1}$, where $f(a_{1})=k_{2}$. Similarly for $g$ with the appropriate modifications.
\item We have that $f(x_{0})Rz_{1}$, but certainly $x_{1}$ witnesses this. For $g$ the same holds.
\end{enumerate}

We now show that it satisfies the strong back condition.
\begin{enumerate}
\item If we consider $f(x_{1}),f(a_{1}),f(a_{0}),f(a_{4}),f(a_{2})$, note that these don't have any proper $\leq$-successor, so the condition can always be instantiated.
\item For $f(x_{1})$, note that this is below $a_{0}$ and $a_{2}$, but these do not have any $R$-successors. So there is no need to close off the strong back condition.
\item Finally we consider $f(x_{0})$. This is below $x_{2},a_{3},a_{4}$. First note that $f(x_{0})\leq z_{2}$, and $z_{2}Rk_{0}$. Then by taking $x_{2}$, we have that $x_{0}\leq x_{2}$, $x_{2}Ra_{0}$, and $f(a_{0})=k_{0}$. Similarly for each of the other nodes, and similarly for $g$, with appropriate modifications.
\end{enumerate}
We thus obtain that $f$ is a $\mathbf{FS}$-morphism, as desired.
\end{proof}

\begin{theorem}\label{thm: failure of coamalgamation}
    The class of $\mathbf{FS}$-spaces is not closed under co-amalgamation.
\end{theorem}
\begin{proof}
    By Lemma \ref{lem: All the frames are IK}, $Z,X,Y$ are all $\mathbf{FS}$-spaces, and by Lemma \ref{lem: fs morphisms}, the maps are $\mathbf{FS}$-morphisms. Assume towards a contradiction that $W$ is a coamalgam of $(Z,X,Y)$, and let $p_{1}\colon W\to X$ and $p_{2}\colon W\to Y$ be surjective $\mathbf{FS}$-morphisms; note that $W$ is an $\mathbf{FS}$-frame by Lemma \ref{lem: FS spaces are FS frames}.\eqref{eq: infinite part}, and hence, satisfies conditions $\mathbf{F1}$ and $\mathbf{F2}$. Let $w_{0}\in W$ be such that $p_{1}(w_{0})=x_{0}$. By the $\leq$-p-morphism condition, then there is some world $w_{0}\leq w_{2}$ such that $p_{1}(w_{2})=x_{2}$. By the $R$-weak back condition, there is some $w_{1}$ such that $w_{0}Rw_{1}$ and $p_{1}(w_{1})\geq x_{1}$. Since the map preserves the $R$-relation, $p_{1}(w_{1})=a_{0}$ or $p_{1}(w_{1})=a_{2}$, would yield  $x_{0}Ra_{0}$ or $x_{0}Ra_{2}$, a contradiction. Thus,  $p_{1}(w_{1})=x_{1}$. Since the diagram commutes, we have that $p_{2}(w_{0})=y_{0}$, $p_{2}(w_{2})=y_{2}$ and $p_{2}(w_{1})=y_{1}$.

    Since $W$ is an $\mathbf{FS}$-frame, there is  $i\in W$  such that $w_{2}Ri$ and $w_{1}\leq i$. Then, as $p_{1}$ preserves both relations, we must have $p_{1}(i)=a_{0}$ and so $fp_{1}(i)=k_{0}$, by construction. On the other hand, for the same reason $p_{2}(i)=b_{2}$ and so $gp_{2}(i)=k_{2}$ -- a contradiction, since we assume the diagram commutes. 
\end{proof}

\begin{theorem}\label{thm:FS_no_amal}
    The Fischer-Sevi logic $\mathsf{IK}$ does not have the Craig interpolation property.
\end{theorem}
\begin{proof}
    The result follows immediately from Theorem~\ref{thm: failure of coamalgamation} and   
    Corollary~\ref{cor: CIP implies coamalgamation}.
\end{proof}

\begin{remark}
We make some remarks on related interpolation properties:
\begin{enumerate}
    \item We note that $\mathbf{FS}$-algebras are easily seen to have the Congruence Extension Property (CEP). By a classic result (see e.g. \cite[Theorem 27]{metcalfesuperinterpolation}) if $\mathbf{K}$ has the CEP, then it has amalgamation if and only if it has \emph{deductive interpolation}. Hence we, in fact, conclude that $\mathsf{IK}$ also fails to have deductive interpolation.
    \item Since uniform interpolation and Lyndon interpolation are both strengthenings of deductive or Craig interpolation, those properties will likewise fail to hold for $\mathsf{IK}$.
\end{enumerate}
\end{remark}

We now consider some related systems:
\begin{itemize}
    \item $\mathsf{IKT}\coloneqq \mathsf{IK}\oplus(\Box p\rightarrow p)\oplus (p\rightarrow \Diamond p)$.
    \item $\mathsf{IK4}=\mathsf{IK}\oplus (\Box p\rightarrow\Box\Box p)\oplus (\Diamond\Diamond p\rightarrow\Diamond p)$.
    \item $\mathsf{IS4}=\mathsf{IKT}\oplus (\Box p\rightarrow\Box\Box p)\oplus (\Diamond\Diamond p\rightarrow\Diamond p)=\mathsf{IK4}\oplus (\Box p\rightarrow p)\oplus (p\rightarrow \Diamond p)$;
    \item $\mathsf{IGL}$ \cite{IGLlogicdasmarinvan} is the set of validities of the class of $\mathbf{FS}$-frames where $R$ is transitive, and $\leq\circ R$ has no infinite paths, i.e., there is no infinite sequence $x_{1}\leq y_{1}Rx_{2}\leq y_{2}Rx_{3}\dots$.
\end{itemize}

The following are analogous to the classical completeness results:

\begin{theorem}\cite[pp.56]{Simpson1994}
    For $(X,\leq,R)$ an $\mathbf{FS}$-frame, we have:
    \begin{enumerate}
        \item $X$ is an $\mathbf{FST}$-frame if and only if $R$ is reflexive;
        \item $X$ is an $\mathbf{FSK4}$-frame if and only if $R$ is transitive;
        \item $X$ is an $\mathbf{FS4}$-frame if and only if $R$ is a preorder (reflexive and transitive).
    \end{enumerate}
\end{theorem}

We can now see the following:

\begin{theorem}
    The intuitionistic modal logics $\mathsf{IKT}$, $\mathsf{IK4}$, $\mathsf{IS4}$ and $\mathsf{IGL}$ do not have the Craig interpolation property.
\end{theorem}
\begin{proof}
    For $\msf{IK4}$ and $\msf{IGL}$, observe that the $\mbf{FS}$-frames $X, Y, Z$ defined above are $\msf{IK4}$ and $\msf{IGL}$ frames: the frames are obviously transitive, and there are no infinite paths since there are no loops with respect to the $R$-relation. We have shown that the co-V-formation does not amalgamate in the class of $\mbf{FS}$-frames, so it also does not amalgamate in a smaller class of frames. Therefore, again by Corollary \ref{cor: CIP implies coamalgamation}, $\mathsf{IK4}$ and $\mathsf{IGL}$ do not have Craig interpolation.

    For $\msf{IS4}$ and $\msf{IKT}$, consider the frames $X', Y', Z'$ obtained from $X, Y, Z$ by making the relation $R$ reflexive. These are still $\mbf{FS}$-frames (because $\mbf{F1}$ and $\mbf{F2}$ become trivial if $x=y$). We check that the maps $f$ and $g$ are still $\mbf{FS}$-morphisms. The first two conditions are immediate. The weak back condition is also straightforward (if $z = f(x)$, take $x' = x$). For the strong back condition, assume that $f(x)\le m$ and $mRm$. Because $f$ is a $\le$-p-morphism, there is some $x\le x'$ such that $f(x') = m$. We also have $x'Rx'$ and $f(x')\le m$, which shows that the strong back condition holds. Therefore, the frames $X', Y', Z'$ are valid $\msf{IS4}$ and $\msf{IKT}$ frames, and the morphisms $f,g$ are valid $\mathbf{FS}$-morphisms. An argument identical to the one in Theorem \ref{thm:FS_no_amal} shows that the co-V-formation $(Z, X, Y)$ cannot have an amalgam, thus completing the proof.
\end{proof}

\section{Conclusions}

In this paper, we have shown that the Fischer-Servi logic $\mathsf{IK}$ and many of its related systems studied in the literature fail to have the Craig Interpolation Property (CIP), by demonstrating that the corresponding classes of algebras do not have the amalgamation property. This naturally raises a number of follow-up questions.

Most importantly, which extensions of $\mathsf{IK}$ do enjoy the CIP? This question can, in turn, be divided into several subquestions. First, which axioms involving $\Box$ and $\Diamond$ yield logics with the CIP?

Another natural direction is to ask whether Fischer–Servi logics over stronger intuitionistic bases exhibit the same behaviour. For example, one could consider adding the G\"odel–Dummett axiom $(p \rightarrow q) \vee (q \rightarrow p)$. We note that this axiom fails in the kind of counterexample provided here, as it forces the underlying frames to be linear. Thus, the counterexample in this paper does not apply to these logics, leaving the study of interpolation for them as an interesting open question.

\paragraph{Acknowledgments} We would like to thank Iris van der Giessen and Marianna Girlando for many interesting discussions on interpolation for intuitionistic modal logics, which  motivated us to  initiate this project. We also thank Guram Bezhanishvili for sharing his recollections of his work on interpolation in the early 2000s, and Qian Chen for helpful suggestions. We are also grateful to Balder ten Cate and Rosalie Iemhoff for fruitful discussions on interpolation in intuitionistic and modal logics.

\printbibliography[
    heading=bibintoc,
    title={Bibliography}
]

@article{Aguilera2025,
	title        = {Intuitionistic G\"{o}del-L\"{o}b without Sharps},
	author       = {Aguilera,  Juan P. and Pacheco,  Leonardo},
	year         = {2025},
	month        = sep,
	journal      = {ACM Transactions on Computational Logic},
	publisher    = {Association for Computing Machinery (ACM)},
	volume       = {26},
	number       = {4},
	pages        = {1–14},
	issn         = {1557-945X},
	url          = {http://dx.doi.org/10.1145/3748649}
}

@article{Bezhanishvili1998,
	title        = {Varieties of Monadic Heyting Algebras. Part I},
	author       = {Bezhanishvili,  Guram},
	year         = {1998},
	month        = nov,
	journal      = {Studia Logica},
	publisher    = {Springer Science and Business Media LLC},
	volume       = {61},
	number       = {3},
	pages        = {367–402},
	issn         = {1572-8730},
	url          = {http://dx.doi.org/10.1023/A:1005073905902}
}

@article{Bezhanishvili1999,
	title        = {Varieties of Monadic Heyting Algebras Part II: Duality Theory},
	author       = {Bezhanishvili,  Guram},
	year         = {1999},
	month        = jan,
	journal      = {Studia Logica},
	publisher    = {Springer Science and Business Media LLC},
	volume       = {62},
	number       = {1},
	pages        = {21–48},
	issn         = {1572-8730},
	url          = {http://dx.doi.org/10.1023/A:1005173628262}
}

@article{Bezhanishvili2000,
	title        = {Varieties of Monadic Heyting Algebras. Part III},
	author       = {Bezhanishvili,  Guram},
	year         = {2000},
	month        = feb,
	journal      = {Studia Logica},
	publisher    = {Springer Science and Business Media LLC},
	volume       = {64},
	number       = {2},
	pages        = {215–256},
	issn         = {1572-8730},
	url          = {http://dx.doi.org/10.1023/A:1005285631357}
}

@book{Blackburn2002-fd,
	title        = {Modal logic},
	author       = {Blackburn, Patrick and de Rijke, Maarten and Venema, Yde},
	year         = {2002},
	publisher    = {Cambridge University Press},
	address      = {Cambridge, England},
	volume       = {Cambridge tracts in theoretical computer science},
	number       = {53}
}

@incollection{BtCI26,
	title        = {Six Proofs of Interpolation for the Modal Logic K},
	author       = {Nick Bezhanishvili and Balder ten Cate and Rosalie Iemhoff},
	year         = {2026},
	booktitle    = {Theory and Applications of Craig Interpolation},
	publisher    = {Ubiquity Press},
	editor       = {Balder ten Cate and  Jean Christoph Jung and Patrick Koopmann and Christoph Wernhard and Frank Wolter}
}

@book{Chagrov1997-cr,
	title        = {Modal Logic},
	author       = {Chagrov, Alexander and Zakharyaschev, Michael},
	year         = {1997},
	publisher    = {Clarendon Press},
	address      = {Oxford, England},
	series       = {Oxford Logic Guides}
}

@incollection{CZERMAK1975381,
	title        = {Interpolation Theorem for Some Modal Logics},
	author       = {J. Czermak},
	year         = {1975},
	booktitle    = {Logic Colloquium '73},
	publisher    = {Elsevier},
	series       = {Studies in Logic and the Foundations of Mathematics},
	volume       = {80},
	pages        = {381--393},
	issn         = {0049-237X},
	url          = {https://www.sciencedirect.com/science/article/pii/S0049237X08719578},
	editor       = {H.E. Rose and J.C. Shepherdson}
}

@book{Esakiach2019HeyAlg,
	title        = {Heyting algebras},
	author       = {Esakia, L.},
	year         = {2019},
	publisher    = {Springer, Cham},
	series       = {Trends in Logic---Studia Logica Library},
	volume       = {50},
	note         = {English translation of the original 1985 book},
	translator   = {A.~Evseev},
	editor       = {Bezhanishvili, G. and Holliday, W. H.}
}

@article{fischerservi1984,
	title        = {Axiomatizations for some intuitionistic modal logics},
    author={Giséle Fischer-Servi},
	year         = {194},
	journal      = {Rendiconti del Seminario Matematico - PoliTO},
	volume       = {42},
	number       = {3},
	pages        = {179--194}
}

@book{Font2016-dk,
	title        = {Abstract algebraic logic -- An introductory textbook},
	author       = {Font, Josep Maria},
	year         = {2016},
	publisher    = {College Publications},
	address      = {London}
}

@inbook{Fre2024,
	title        = {Mechanised Uniform Interpolation for Modal Logics K,  GL,  and iSL},
	author       = {Férée,  Hugo and Giessen,  Iris van der and Gool,  Sam van and Shillito,  Ian},
	year         = {2024},
	booktitle    = {Automated Reasoning},
	publisher    = {Springer Nature Switzerland},
	pages        = {43–60},
	issn         = {1611-3349},
	url          = {http://dx.doi.org/10.1007/978-3-031-63501-4_3}
}

@incollection{Fus26,
	title        = {Interpolation in Non-Classical Logics},
	author       = {Wesley Fussner},
	year         = {2026},
	booktitle    = {Theory and Applications of Craig Interpolation},
	publisher    = {Ubiquity Press},
	editor       = {Balder ten Cate and  Jean Christoph Jung and Patrick Koopmann and Christoph Wernhard and Frank Wolter}
}

@inproceedings{Girlando2023,
	title        = {Intuitionistic S4 is decidable},
	author       = {Girlando,  Marianna and Kuznets,  Roman and Marin,  Sonia and Morales,  Marianela and Straßburger,  Lutz},
	year         = {2023},
	month        = jun,
	booktitle    = {2023 38th Annual ACM/IEEE Symposium on Logic in Computer Science (LICS)},
	publisher    = {IEEE},
	pages        = {1–13},
	url          = {http://dx.doi.org/10.1109/LICS56636.2023.10175684}
}

@book{GM05,
	title        = {Interpolation and definability: modal and intuitionistic logics},
	author       = {Gabbay, Dov M and Maksimova, Larisa},
	year         = {2005},
	publisher    = {Oxford University Press},
	volume       = {1}
}

@article{Halmos1954-1956,
	title        = {Algebraic logic, I. Monadic boolean algebras},
	author       = {Halmos, Paul R.},
	year         = {1954},
	journal      = {Compositio Mathematica},
	publisher    = {Kraus Reprint},
	volume       = {12},
	pages        = {217--249},
	url          = {http://eudml.org/doc/88816},
	language     = {eng}
}

@book{HandbookInterp26,
	title        = {Theory and Applications of Craig Interpolation},
	author       = {Balder ten Cate and Jean Christoph Jung and  Patrick Koopmann and Christoph Wernhard and Frank Wolter},
	year         = {2026},
	publisher    = {Ubiquity Press},
	url          = {https://cibd.bitbucket.io/taci/}
}

@article{Iemhoff2019,
	title        = {Uniform interpolation and the existence of sequent calculi},
	author       = {Iemhoff,  Rosalie},
	year         = {2019},
	month        = nov,
	journal      = {Annals of Pure and Applied Logic},
	publisher    = {Elsevier BV},
	volume       = {170},
	number       = {11},
	pages        = {102711},
	issn         = {0168-0072},
	url          = {http://dx.doi.org/10.1016/j.apal.2019.05.008}
}

@inbook{Iemhoff2024,
	title        = {Proof Theory for Lax Logic},
	author       = {Iemhoff,  Rosalie},
	year         = {2024},
	booktitle    = {Dick de Jongh on Intuitionistic and Provability Logics},
	publisher    = {Springer International Publishing},
	pages        = {203–229},
	issn         = {2211-2766},
	url          = {http://dx.doi.org/10.1007/978-3-031-47921-2_8}
}

@inproceedings{IGLlogicdasmarinvan,
	title        = {Intuitionistic G\"{o}del-L\"{o}b Logic,  à la Simpson: Labelled Systems and Birelational Semantics},
	author       = {Das,  Anupam and van der Giessen,  Iris and Marin,  Sonia},
	year         = {2024},
	journal      = {LIPIcs,  Volume 288,  CSL 2024},
	publisher    = {Schloss Dagstuhl – Leibniz-Zentrum f\"{u}r Informatik},
	volume       = {288},
	pages        = {22:1--22:18},
	url          = {https://drops.dagstuhl.de/entities/document/10.4230/LIPIcs.CSL.2024.22},
	copyright    = {Creative Commons Attribution 4.0 International license},
	language     = {en}
}

@article{Maks1979,
	title        = {Interpolation theorems in modal logics and amalgamable varieties of topological Boolean algebras},
	author       = {Maksimova, Larisa},
	year         = {1979},
	journal      = {Algebra and Logic},
	publisher    = {Springer},
	volume       = {18},
	number       = {5},
	pages        = {348--370}
}

@article{Maksimova1977,
	title        = {Craig’s theorem in superintuitionistic logics and amalgamable varieties of pseudo-boolean algebras},
	author       = {Larisa Maksimova},
	year         = {1977},
	month        = nov,
	journal      = {Algebra and Logic},
	publisher    = {Springer Science and Business Media LLC},
	volume       = {16},
	number       = {6},
	pages        = {427–455},
	issn         = {1573-8302},
	url          = {http://dx.doi.org/10.1007/BF01670006}
}

@inbook{Marx1998,
	title        = {Interpolation in Modal Logic},
	author       = {Marx,  Maarten},
	year         = {1998},
	booktitle    = {Algebraic Methodology and Software Technology},
	publisher    = {Springer Berlin Heidelberg},
	pages        = {154–163},
	issn         = {0302-9743},
	url          = {http://dx.doi.org/10.1007/3-540-49253-4_13}
}

@article{Marx1998-MAAFOI,
	title        = {Failure of Interpolation in Combined Modal Logics},
	author       = {Maarten Marx and Carlos Areces},
	year         = {1998},
	journal      = {Notre Dame Journal of Formal Logic},
	publisher    = {Duke University Press},
	volume       = {39},
	number       = {2},
	pages        = {253--273}
}

@incollection{Met26,
	title        = {Interpolation and Amalgamation},
	author       = {George Metcalfe},
	year         = {2026},
	booktitle    = {Theory and Applications of Craig Interpolation},
	publisher    = {Ubiquity Press},
	editor       = {Balder ten Cate and  Jean Christoph Jung and Patrick Koopmann and Christoph Wernhard and Frank Wolter}
}

@misc{metcalfesuperinterpolation,
	title        = {Interpolation and Amalgamation},
	author       = {Metcalfe,  George},
	year         = {2025},
	publisher    = {arXiv},
	url          = {https://arxiv.org/abs/2512.00924},
	copyright    = {Creative Commons Attribution 4.0 International}
}

@inbook{Muravitsky2014,
	title        = {Logic KM: A Biography},
	author       = {Muravitsky,  Alexei},
	year         = {2014},
	booktitle    = {Leo Esakia on Duality in Modal and Intuitionistic Logics},
	publisher    = {Springer Netherlands},
	pages        = {155–185},
	issn         = {2211-2766},
	url          = {http://dx.doi.org/10.1007/978-94-017-8860-1_7}
}

@periodical{paivaartemov,
	title        = {Journal of Applied Logic},
	title        = {Intuitionistic Modal Logic 2017},
	year         = {2021},
	volume       = {8},
	editor       = {Valeria de Paiva and Sergey Artemov}
}

@incollection{PalmigianoDualities2004,
	title        = {Dualities for some intuitionistic modal logics},
	author       = {Palmigiano, Alessandra},
	year         = {2004},
	booktitle    = {Liber Amicorum for Dick de Jongh},
	publisher    = {Institute for Logic, Language and Computation},
	pages        = {151--167}
}

@misc{paperonnondistributivedualities,
	title        = {Superamalgamation for modal lattices via non-distributive dualities},
	author       = {Almeida,  Rodrigo Nicolau and Bezhanishvili,  Nick and Lemal,  Simon},
	year         = {2026},
	publisher    = {arXiv},
	url          = {https://arxiv.org/abs/2602.20380},
	copyright    = {Creative Commons Attribution 4.0 International}
}

@phdthesis{Simpson1994,
	title        = {The Proof Theory and Semantics of Intuitionistic Modal Logic},
	author       = {Simpson, Alex K.},
	year         = {1994},
	school       = {University of Edinburgh}
}

@article{vanderGiessen2021,
	title        = {Sequent Calculi for Intuitionistic G\"{o}del–L\"{o}b Logic},
	author       = {van der Giessen,  Iris and Iemhoff,  Rosalie},
	year         = {2021},
	month        = may,
	journal      = {Notre Dame Journal of Formal Logic},
	publisher    = {Duke University Press},
	volume       = {62},
	number       = {2},
	issn         = {0029-4527},
	url          = {http://dx.doi.org/10.1215/00294527-2021-0011}
}

@misc{vandergiessenshilito,
	title        = {Uniform interpolation with constructive diamond},
	author       = {van der Giessen,  Iris and Shillito,  Ian},
	year         = {2026},
	publisher    = {arXiv},
	url          = {https://arxiv.org/abs/2602.16880},
	copyright    = {Creative Commons Attribution Non Commercial No Derivatives 4.0 International}
}

@phdthesis{vandergiessenthesis,
	title        = {Uniform Interpolation and Admissible Rules: Proof-theoretic investigations into (intuitionistic) modal logics},
	author       = {van der Giessen,  Iris},
	url          = {http://dx.doi.org/10.33540/1486},
	school       = {Utrecht University Library}
}

@inproceedings{Wolter1997IntuitionisticML,
	title        = {Intuitionistic Modal Logics as Fragments of Classical Bimodal Logics},
	author       = {Frank Wolter and Michael Zakharyaschev},
	year         = {1997},
	url          = {https://api.semanticscholar.org/CorpusID:6398005}
}

@inbook{Wolter1999,
	title        = {Intuitionistic Modal Logic},
	author       = {Wolter,  Frank and Zakharyaschev,  Michael},
	year         = {1999},
	booktitle    = {Logic and Foundations of Mathematics},
	publisher    = {Springer Netherlands},
	pages        = {227–238},
	url          = {http://dx.doi.org/10.1007/978-94-017-2109-7_17}
}

\section{Appendix: Some folklore proofs}

As noted in the introduction, some results we would expect to be available in the literature from an algebraic and dual theoretic point of view seem to be missing. In this appendix we provide some short proofs of these results, employing some standard techniques.

\subsection{Logics, Algebras and Duality}

We start with the language $\mathcal{L}_{\Box}$, and consider the following systems:
\begin{enumerate}
    \item $\mathsf{iK}=\mathsf{IPC}\oplus (\Box(p\wedge q)\leftrightarrow (\Box p\wedge \Box q)\oplus \Box\top$.
    \item $\mathsf{iT}=\mathsf{iK}\oplus \Box p\rightarrow p$;
    \item $\mathsf{iK4}=\mathsf{iK}\oplus \Box p\rightarrow \Box\Box p$;
    \item $\mathsf{iS4}=\mathsf{iK4}\oplus \Box p\rightarrow p$.
\end{enumerate}

We denote by $\mathbf{iL}$ the category of $\Box$-modal Heyting algebras, i.e., pairs $(H,\Box)$ where $H$ is a Heyting algebra and $\Box$ is a modal operation, satisfying the axioms from $L$. By the basic algebraic completeness, we have completeness of the logics $\mathsf{iL}$ with respect to $\mathsf{iL}$-algebras. The dualities by Palmigiano \cite{PalmigianoDualities2004} also cover these logics:

\begin{definition}
    A tuple $(X,\leq,R,\tau)$ is called an \emph{iK-space} if $(X,\leq,\tau)$ is an Esakia space, and satisfies the compatibility condition $R={\leq}\circ R\circ{\leq}$. A tuple $(X,\leq,R)$ in these conditions is called an \emph{iK-frame}.
\end{definition}

\begin{definition}
    A p-morphism $f\colon (X,R)\to (Y,R)$ between iK-spaces is called an \emph{iK}-morphism if $f$ is also an $R$-bounded morphism, i.e., whenever $xRy$ then $f(x)Rf(y)$, and whenever $f(x)Ry$ then there is some $y'$ such that $xRy'$ and $f(y')=y$.
\end{definition}

We denote by $\mathbf{iK}_{Sp}$ the category of $\mathbf{iK}$-spaces with their morphisms. Then by Palmigiano, we have:

\begin{theorem}\cite[Theorem 6.1.11]{PalmigianoDualities2004}
    The categories $\mathbf{iK}_{\mathrm{alg}}$ and $\mathbf{iK}_{\mathrm{sp}}$ are dually equivalent.
\end{theorem}

This is obtained essentially by combining Esakia duality \cite{Esakiach2019HeyAlg} with Jónnson-Tarski duality (see e.g. \cite{Blackburn2002-fd}). Explicitly, given $(H,\Box)$ a $\Box$-modal Heyting algebra, if $X=\mathsf{Spec}(H)$ is its spectrum, we define
\begin{equation*}
    xRy \iff \forall a\in H,  \Box a\in x\rightarrow  a\in y.
\end{equation*}

In turn, given an iK-space $(X,\leq,R,\tau)$, we consider for $U$ a (clopen) upset:
\begin{equation*}
    \Box_{R}U=\{x\in X : \forall y(xRy \rightarrow y\in U\}.
\end{equation*}
Note that the compatibility condition ensures that this is an upset. Note also that this definition makes sense for arbitrary upsets, and we thus obtain that $(\mathsf{Up}(X),\Box_{R})$ is then an iK-algebra, always. The following correspondences, showing that the above duality restricts appropriately, follow from more general principles, though we sketch them out here for completeness:

\begin{definition}
    Given an iK-space $(X,\leq,R,\tau)$ and a logic $\mathsf{iL}$ we say that it is an $\mathbf{iL}$-space if $\mathsf{ClopUp}(X)$ is an $\mathbf{iL}$-algebra. Given an iK-frame, $(X,\leq,R)$, we say that it is an $\mathbf{iL}$-frame if $(\mathsf{Up}(X),\Box_{R})$ is an $\mathbf{iL}$-algebra.
\end{definition}

\begin{proposition}\label{prop: equivalence of algebras and some spaces}
    The following categories are equivalent:
    \begin{enumerate}
        \item \label{eq: reflexive} The category $\mathbf{iT}_{\mathrm{alg}}$ and $\mathbf{iT}_{\mathrm{sp}}$ of iK-spaces with $R$ a reflexive relation.
        \item \label{eq: transitive} The category $\mathbf{iK4}_{\mathrm{alg}}$ and $\mathbf{iK4}_{\mathrm{sp}}$ of iK-spaces with $R$ a transitive relation.
        \item The category $\mathbf{iS4}_{\mathrm{alg}}$ and $\mathbf{iS4}_{\mathrm{sp}}$ of iK-spaces with $R$ a reflexive and transitive relation.
    \end{enumerate}
\end{proposition}
\begin{proof}
    \eqref{eq: reflexive}. By the standard soundness, $\Box p\rightarrow p$ is valid in a frame with a reflexive relation. Assume then that $\neg (xRx)$. By our definition, this means that for some $a\in H$, $\Box a\in x$ and $a\notin x$. If $\Box a\rightarrow a=1$ then since $x$ is a filter, $\Box a\rightarrow a\in x$; then we would have $\Box a\wedge (\Box a\rightarrow a)\in x$ since $x$ is closed under $\wedge$. But it follows (essentially by Modus Ponens) that $\Box a\wedge (\Box a\rightarrow a)\leq a$, so then since $x$ is upwards closed, a contradiction.

    \eqref{eq: transitive}. The soundness is again obvious. Assume that $xRy$ and $yRz$ but not $xRz$; then by definition there is some $a\in H$ such that $\Box a\in x$ and $a\notin z$; by the same arguments as in the previous case, if $\Box a\rightarrow \Box\Box a$ was valid, then $\Box\Box a\in x$ would imply by definition $\Box a\in y$ and so $a\in z$, a contradiction.

    The case for (3) is a combination of (1) and (2).
\end{proof}

\begin{proposition}\label{prop: closure under frames}
    Given an iK-frame $(X,\leq,R)$ we have that:
    \begin{enumerate}
        \item $X$ is a $\mathbf{iT}$-frame if and only if $R$ is reflexive.
        \item $X$ is a $\mathbf{iK4}$-frame if and only if $R$ is transitive.
        \item $X$ is an $\mathbf{iS4}$-frame if and only if $R$ is reflexive and transitive.
    \end{enumerate}
\end{proposition}
\begin{proof}
    Follows by very similar arguments as in the classical case, see \cite[Chapter 3]{Blackburn2002-fd}.
\end{proof}

\subsection{Craig interpolation from Superamalgamation}

\begin{definition}
    Given $\mathbf{M}$ a variety of algebras, $(A,B_{1},B_{2},h_{1},h_{2})$ a $V$-formation, we say that $(C,p_{1},p_{2})$, an amalgam, is a \emph{superamalgam} if:
    \begin{enumerate}
        \item Whenever $p_{1}(b_{1})\leq p_{2}(b_{2})$ for $b_{i}\in B_{i}$, then there is some $a\in A$ such that $b_{1}\leq h_{1}(a)$ and $h_{2}(a)\leq b_{2}$.
        \item Whenever $p_{2}(b_{2})\leq p_{1}(b_{1})$ for $b_{i}\in B_{i}$, then there is some $a\in A$ such that $b_{2}\leq h_{2}(a)$ and $h_{1}(a)\leq b_{1}$.
    \end{enumerate}
    We say that $\mathbf{M}$ has the \emph{superamalgamation property} if for every V-formation, there is a superamalgam in $\mathbf{M}$.
\end{definition}

The following proceeds by the same standard argument employed in \cite{Maks1979}:

\begin{proposition}
    Assume that $\mathsf{iL}$ is a normal modal logic above $\mathsf{iK}$. If $\mathbf{iL}$-algebras have the superamalgamation property, then $\mathsf{iL}$ has the Craig interpolation property.
\end{proposition}

We thus seek to establish superamalgamation for these varieties. For this we make use of a standard technique \cite{BtCI26}. Let $(X,Y_{1},Y_{2},f,g)$ a co-V-formation of iK-spaces, with $f$, and $g$ surjective iK-morphisms. Define:
\begin{equation*}
        \mathsf{Pb}(f,g)=\{(y_{1},y_{2})\in Y_{1}\times Y_{2} : f(y_{1})=g(y_{2})\}.
\end{equation*}
We define $\leq$ and $R$ pointwise on this structure. The following captures the main property we will exploit of this:

\begin{proposition}\label{prop: CIP from pullback}
    Assume that $\mathsf{iL}\in \{\mathsf{iT},\mathsf{iK4},\mathsf{iS4}\}$. Suppose that whenever $(A,B_{1},B_{2},h_{1},h_{2})$, is a V-formation of $\mathbf{iL}$-algebras, if we consider the dual co-V-formation $(X,Y_{1},Y_{2},f,g)$, we have that $\mathsf{Pb}(f,g)$ is an $\mathbf{iL}$-frame. Then $\mathsf{iL}$ has the Craig interpolation property.
\end{proposition}
\begin{proof}
    To show that the algebra $\mathsf{Up}(\mathsf{Pb}(f,g))$ forms an amalgam, it suffices to show, by duality,  that $\pi_{Y_{1}}$ and $\pi_{Y_{2}}$ are surjective, $\leq$-p-morphisms and $R$-bounded morphisms, and that the diagram commutes. It is trivial to see by the definition of $\mathsf{Pb}(f,g)$ that the diagram commutes. The fact that it is surjective follows by the arguments as in the case of intuitionistic logic \cite{Maksimova1977}, and for similar reasons, that $\pi_{Y_{i}}$ are $\leq$-p-morphisms. The fact that they are $R$-bounded morphisms also follows the classical modal logic case (see e.g., \cite{BtCI26}). The fact that it superamalgamates is similar to the proof given in \cite{paperonnondistributivedualities}. Finally, the fact that the resulting algebra $\mathsf{Up}(\mathsf{Pb}(f,g))$ is an $\mathbf{iL}$-algebra follows from the fact that $\mathsf{Pb}(f,g)$ is an $\mathsf{iL}$-frame, by definition.
\end{proof}

\begin{corollary}
    The logics $\mathsf{iT},\mathsf{iK4},\mathsf{iS4}$ have the CIP.
\end{corollary}
\begin{proof}
    By Proposition \ref{prop: CIP from pullback}, we just need to show that $\mathsf{Pb}(f,g)$ is an $\mathsf{iL}$-frame. We do this for $\mathsf{iT}$, the other cases being similar. Note that by Proposition \ref{prop: equivalence of algebras and some spaces}, the relations $R$ in the spaces $(X,Y_{1},Y_{2},f,g)$ will all be reflexive. Then the relation defined pointwise in $Y_{1}\times Y_{2}$ will likewise be reflexive, and since $\mathsf{Pb}(f,g)$ is a subposet, and sub-relations of reflexive relations are reflexive, this will also be reflexive, i.e., $(\mathsf{Pb}(f,g),R)$ will be a reflexive frame. Since products and substructures of transitive structures are again transitive, this similarly goes through for $\mathsf{iK4}$ and $\mathsf{iS4}$.
\end{proof}

\end{document}